\documentclass[a4paper,12pt]{amsart} 
\usepackage[utf8x]{inputenc}
\usepackage{amsmath, amssymb, amsthm, latexsym, mathrsfs, graphicx}
\usepackage{tikz-cd}
\usepackage[margin=1in]{geometry}

\usepackage{hyperref}
\hypersetup{
	colorlinks=true,
	linkcolor=blue,
	filecolor=magenta,      
	urlcolor=cyan,
}

\newcommand{\Gal}{\rm Gal}
\newcommand{\GL}{\rm GL}

\newcommand{\GLnF}{{\rm GL}_n(F)}

\newcommand{\id}{\rm id}
\newcommand{\Ind}{\rm Ind}

\newcommand{\Irr}{{\rm Irr}}

\newcommand{\lra}{\longrightarrow}

\newcommand{\mc}{\mathcal}

\newcommand{\slope}{{\rm sl}}

\newcommand{\Sw}{{\rm Sw}}

\newtheorem{theorem}{Theorem}[section] 
\newtheorem{lemma}[theorem]{Lemma}
\newtheorem{corollary}[theorem]{Corollary}
\newtheorem{proposition}[theorem]{Proposition}
\newtheorem{definition}[theorem]{Definition}

\newtheorem{remark}{Remark}

\numberwithin{question}{section}
\numberwithin{equation}{section}

\begin{document}
	\title{On the depth of the adjoint representations}
	\author{Arindam Jana and Amiya Mondal}
	\date{\today}
	\address{Indian Institute of Science Education and Research Berhampur, Odisha 760003, India.}	
	\email{arindamj@iiserbpr.ac.in, arindamjana076@gmail.com}
	\email{amiya@iiserbpr.ac.in}
	
	

	\begin{abstract}
		Let $F$ be a non-archimedean local field of odd residual characteristic $p$. The depth of a smooth representation of ${\rm GL}_n(F)$ is an invariant of Local Langlands Correspondence (LLC). The analogous notion on the Galois side of LLC is known as the slope of a local Galois representation. The slope is well related to the Swan conductor for irreducible Galois representations, whereas its behavior is subtle for reducible Galois representations. In this article, we provide an explicit formula for the slope of the adjoint of a Carayol representation of the local Galois group.
	\end{abstract}
	
	\maketitle
	
	

	\section{Introduction}
	Let $F$ be a non-archimedean local field of odd residual characteristic $p$. Let $\bf G$ be a connected reductive group defined over $F$ with the group of $F$-rational points $G:={\bf G} (F)$. The (conjectural) Local Langlands Correspondence (LLC) for the group $G$ is a finite to one map
	\[\Irr(G)\lra\Phi(G);~~\pi\mapsto\phi_\pi,\]
	where $\Irr(G)$ denotes the set of equivalence classes of all irreducible smooth representations of $G$ and $\Phi(G)$ denotes the set of equivalence classes of all Langlands parameters of $G$.  
	
	An important aspect in the smooth representation theory of $p$-adic groups is the study of the conductor and the depth of smooth representations. It is natural to ask if the LLC preserves conductor and depth.
	
	The conductor of a smooth representation is defined via the $\varepsilon$-factor associated to the local functional equation. It is known that the conductor is preserved under LLC for $\GLnF$. The corresponding conductor on the Galois side is known as the Artin conductor. The Swan conductor is also an invariant closely related to the Artin conductor.     
	
    Hence, it is of interest to investigate the behaviour of the Artin and Swan conductor under functoriality which is one of the pillars of the Langlands Program. For given irreducible supercuspidal representations $\pi_i$ of ${\GL}_{n_i}(F)$ for $i=1,2$, the Rankin-Selberg lift $\pi_1\times\pi_2$ to the group ${\GL}_{n_1n_2}(F)$ corresponds to the tensor product $\phi_{\pi_1}\otimes\phi_{\pi_2}$ under the LLC. The pioneering work of Bushnell-Henniart-Kutzko studies the conductor of the Rankin-Selberg lift (see \cite[Theorem 6.5, p. 723]{bus-hen-kut98}). Using this result, the conductor of other functorial lifts, namely, symmetric square, exterior square and Asai lifts have been studied in \cite{ana-mon15,ana-mon16}. Although an explicit conductor formula on the Galois side of the correspondence might be a bit more involved, sharp bound of these invariants have been studied using purely Galois-theoretic techniques (see \cite{bus-hen17,kil19, hen-oi24} for reference).

	The depth $d(\pi)$ of an irreducible smooth representation $\pi$ of a reductive $p$-adic group $G$ was defined by Moy-Prasad \cite{moy-pra94,moy-pra96} in terms of filtrations of its parahoric subgroups. On the other hand, under the (conjectural) local Langlands correspondence (LLC), one can define the depth $d(\phi_\pi)$ of the Langlands parameter $\phi_\pi$ of $\pi$ using the ramification filtrations of the absolute Galois group $\Gal(F_s/F)$ of $F$ (see \cite[p. 18]{abps16}). 
	It is an active area of research whether depth is preserved under the LLC.
	This means, one would like to know if
	\[d(\pi)=d(\phi_\pi),~{\rm for}~ \pi\in\Pi_\phi(G).\] 
    This is proved in many cases. In particular, this is known for the group ${\GL}_n(F)$ \cite[Theorem 2.15, p. 541]{oi23} (originally \cite[2.3.6]{yu09}) and \cite[Proposition 4.2]{abps16}.
	
    Along the same line, one would like to study the behaviour of depth under functorial lifts on ${\GL}_n(F)$. For irreducible supercuspidal representations $\pi_i$ of ${\GL}_{n_i}(F)$ for $i=1,2$, it is known that
        \[d(\pi_1\times\pi_2)\leq\max\{d(\pi_1),d(\pi_2)\}.\]
    Moreover, if $d(\pi_1)\neq d(\pi_2)$, then $d(\pi_1\times\pi_2)=\max\{d(\pi_1),d(\pi_2)\}$ (see \cite[Proposition 3.4, p. 21]{kil19}). 
	
    In this article, we are interested in the depth of the adjoint lift. More precisely, for a Carayol representation $\pi$ of ${\GL}_n(F)$, we find the depth $d(\pi\times\pi^\vee)$, where $\pi^\vee$ denotes the contragredient or smooth dual of $\pi$. The notion of slope $\slope(\rho)$ of a Galois representation $\rho$ is defined in (\cite[\S(3.2.1)]{bus-hen19}, \cite[Definition 3.1, p. 20]{kil19}).
     Comparing this definition with that of depth as in \cite[p. 18]{abps16}), it follows that 
            \[d(\phi_\pi)=\slope(\phi_\pi),\]
    for any Langlands parameter $\phi_\pi$ of $\GLnF$. With this understanding, our treatment is fully on the Galois side of the LLC. From now onwards, we use the notation
    $\slope(\rho)$ for the depth a Galois representation $\rho$.
	
    Let $\widehat{\mathcal{W}}_F$ be the set of equivalence classes of smooth, irreducible complex representations of the Weil group $\mathcal{W}_F$ associated to $F$. A representation $\rho\in\widehat{\mathcal W_F}$ is said to be {\bf Carayol} if $({\rm Sw}(\rho), \dim \, \rho)=1$ \cite[\S 7]{kil19}. The representation $\rho$ is called {\bf epipelagic} if in particular ${\rm Sw}(\rho)=1$.

    It turns out that all Carayol representation are tamely irreducible (\cite[Proposition 7.2]{kil19}, \cite{bus-hen14}), i.e., restriction of $\rho$ on the inertia remains irreducible. However, sometimes in the literature, e.g., in \cite{bus-hen19}, Carayol representations are defined on the class of wildly irreducible representations, i.e., restriction of $\rho$ on the wild inertia is irreducible. Notice that every wildly irreducible representation is tamely irreducible. In this paper we work on the larger class of Carayol representations as in \cite{kil19} which are tamely but not necessarily wildly irreducible.
        

    Let $\rho$ be a Carayol representation of $\mathcal{W}_F$. We find an explicit formula for the slope of the adjoint representation ${\rm Ad}(\rho):=\rho \otimes {\rho}^{\vee}$. Our method uses the ramification filtration of the Weil group $\mc W_F$ and its associated Herbrand functions. The first main theorem of this article is the following.
       \begin{theorem}\label{int thm slope adjoint tame}
		Let $F$ be a $p$-adic field of residue characteristic $p\neq 2$. Let $\rho$ be a Carayol representation of $\mc W_F$ of dimension $mp^r$ where $(p, m)=1$. If $m>1$, then ${\rm sl}(\rho\otimes\rho^\vee)={\rm sl}(\rho)$.
        \end{theorem}

    Using the LLC, it follows from Theorem \ref{int thm slope adjoint tame} that $d(\pi\times\pi^\vee)=d(\pi)$ for a Carayol representation $\pi$ of $\GLnF$ where $n=mp^r$ with $m>1$. On the other hand, if $m=1$, then we have obtained an explicit formula for $d(\pi\times\pi^\vee)$ which shows that $d(\pi\times\pi^\vee)< d(\pi)$. To explain the result, we shift to the Galois side of the LLC.

    Let $\rho$ be a Carayol Galois representation of dimension $p^r$. Before we describe the general result on $\slope(\rho\otimes\rho^\vee)$, let us mention the following theorem which describes the slope of the adjoint of a Carayol representation, wildly induced from a character (see Theorem \ref{thm slope adjoint wild} for details).
    This result lies in the heart of the second main theorem (Theorem \ref{int thm slope carayol}) to follow. For convenience, from now on we denote ${\rm Ind}_{K/F}(~):={\rm Ind}_{\mathcal{W}_K}^{\mathcal{W}_F}(~)$ for an extension $K/F$. 
	\begin{theorem}\label{int thm slope adjoint wild} 
		Let $F$ be a $p$-adic field of residue characteristic $p\neq 2$. Let $\rho$ be a Carayol representation of $\mc W_F$ of dimension $p^r$, $r>1$. Suppose $\rho={\rm Ind}_{K/F}\chi$ for some totally ramified Galois extension $K/F$ such that ${\rm Gal}(K/F)$ is cyclic and a character $\chi$ of $\mc W_K$. Then
		\[{\rm sl}(\rho \otimes \rho^{\vee})={\rm sl}(\rho)-\dfrac{i_0}{p^r},\] 
		where $i_0$ is the smallest lower jump in ${\Gal}(K/F)$.
	\end{theorem}
	
	The central idea of the proof of Theorem \ref{int thm slope adjoint wild} is an explicit understanding of the lower and upper jumps of the Herbrand functions associated to cyclic extensions which has been carried out in Section \ref{sec lemmas cyclic}. Using these techniques for the cyclic extensions, we further extend theorem \ref{int thm slope adjoint wild} for arbitrary extensions.
    \begin{theorem}\label{int thm slope adjoint any wild} 
		Let $F$ be a $p$-adic field of residue characteristic $p\neq 2$. Let $\rho$ be a Carayol representation of $\mc W_F$ of dimension $p^r$, $r>1$. Suppose $\rho={\rm Ind}_{K/F}\chi$ for some totally ramified Galois extension $K/F$ and a character $\chi$ of $\mc W_K$. Then
		\[{\rm sl}(\rho \otimes \rho^{\vee})={\rm sl}(\rho)-\dfrac{i_0}{p^r},\] 
		where $i_0$ is the smallest lower jump in ${\Gal}(K/F)$.
	\end{theorem}

    However, Carayol representations need not be induced from characters in general, for example, epipelagic representations of dimension $p^r$. The good thing is that any Carayol representation can be constructed from a character $\chi$ of $\mc W_K$ for some extension $K/F$ through certain inductions and restrictions (see Section \ref{sec slope gen}). The following theorem describes the slope of the adjoint of any Carayol representation.
	\begin{theorem}\label{int thm slope carayol}
      Let $F$ be a $p$-adic field of residue characteristic $p\neq 2$. Let $\rho$ be a Carayol representation of $\mc W_F$ of dimension $p^r$. Then
       \[{\rm sl}(\rho\otimes \rho^{\vee})={\rm sl}(\rho)-\dfrac{\phi_{K/F}(i_0)}{\dim \, \rho}\, ,\]
      where $i_0$ is the smallest nonzero lower jump in ${\Gal}(K/F)$ for the finite Galois extension $K/F$. 
    \end{theorem}

	Since epipelagic representations form an important subclass of Carayol representations, let us write the slope of adjoint of an epipelagic representation as a corollary to Theorems \ref{int thm slope adjoint tame} and \ref{int thm slope carayol}.
	\begin{corollary}\label{cor slope epipelagic}
		Let $F$ be a $p$-adic field of residue characteristic $p\neq 2$. Let $\rho$ be an epipelagic representation of $\mathcal{W}_F$. 
		\begin{enumerate}
			\item[{\rm(1)}] If $m>1$, then ${\rm sl}(\rho\otimes \rho^{\vee})=\dfrac{1}{\dim\,\rho}$.
			\item[{\rm(2)}]  If $m=1$, then
			\[{\rm sl}(\rho\otimes \rho^{\vee})=\dfrac{1-\phi_{K/F}(i_0)}{\dim\,\rho},\] where $i_0$ is the smallest nonzero lower jump in ${\Gal}(K/F)$ for the finite Galois extension $K/F$.     
		\end{enumerate}
	\end{corollary}
	
	Let us briefly mention the contents of the article. In section \ref{sec slope}, we define the notion of the slope of a Galois representation and prove some basic properties of slope. In Section \ref{sec herbrand}, we describe some generalities on Herbrand functions. In Section \ref{sec slope wild}, we prove Theorems \ref{int thm slope adjoint wild} and \ref{int thm slope adjoint any wild}. Finally in Section \ref{sec slope gen}, we establish the slope formula in general (the proof of Theorems \ref{int thm slope adjoint tame} and \ref{int thm slope carayol}).

    \section{Slope of a Galois representation}\label{sec slope}
	We start with the definition of the slope of a Weil representation. Let $F$ be a  $p$-adic field of residue characteristic $p\neq 2$ and let $\mathcal{W}_F$ denote the associated Weil group. Let $\{\mathcal{W}_F^r\}_{r\in \mathbb{R}_{\geq -1}}$ denote the upper ramification filtration of $\mathcal{W}_F$ with $\mathcal{W}_F^{-1}=\mathcal{W}_F$ and $\mathcal{W}_F^0=\mathcal{I}_F$ denote the inertia subgroup of $\mathcal{W}_F$. For $r\geq 0$, 
	the subgroups $\mathcal{W}_F^{r}$ 
	are profinite, closed and normal in $\mathcal{W}_F$.
	Let 
	$\mathcal{P}_F$ denote the 
	wild inertia 
	of $\mathcal{W}_F$.
	
	Let $\rho$ be a finite dimensional smooth representation of $\mathcal{W}_F$. There exists a finite Galois extension $E$ of $F$ such that $\rho\vert_{\mathcal{W}_F^0}$ factors through the inertia subgroup $G^0$ of $G:={\rm Gal}(E/F)$. Let $\{G_i\}_{i\in \mathbb{Z}_{\geq -1 }}$ and $\{G^r\}_{r\in \mathbb{R}_{\geq -1}}$ be the lower and upper ramification filtrations of $G$ \cite[Chapter IV]{ser79}. We say that $G$ has a lower jump at $l\geq -1$ if $G_{l} \neq G_{l+1}$. Also, we say that $G$ has an
	upper jump at $u\geq -1$ if there exists $\epsilon>0$ such that $G^u \neq G^{u+\epsilon}$ for all $\epsilon>0$.
	\begin{definition}{\rm\cite[(3.2.1)]{bus-hen19}},\cite[Definition 3.1, p. 20]{kil19}\label{slope defn}
		Let $\rho$ be a finite dimensional smooth representation of $\mathcal{W}_F$. Then the slope 
		of $\rho$ is defined by 
		\[{\rm sl}(\rho):= {\rm inf}\{r\in \mathbb{R}_{>0}\mid \rho\big\vert_{\mathcal{W}_F^r}=1\}.\]
	\end{definition}
	In this connection, let us recall the definition of the Swan conductor of the representation $\rho$.
	\begin{definition}\label{def swan}
		Let $(\rho,V)$ be a finite dimensional smooth representation of $\mathcal{W}_F$. Then the Swan conductor
		of $\rho$ is defined by 
		\[{\rm Sw}(\rho):=\int_{0}^\infty\,{\rm dim}\, (V/{V^u})\,du,\]
		where $V^u$ denotes the subspace of $\mathcal W_F^u$-invariant vectors.
	\end{definition}
	For any representation $\rho$ of $\mathcal{W}_F$, we have ${\rm sl}(\rho)\geq \frac{{\rm Sw}(\rho)}{\dim \rho}$. In particular, when $\rho$ is irreducible, it follows that ${\rm Sw}(\rho)={\rm sl}(\rho)\cdot {\rm dim}\, (\rho)$ \cite[Theorem 3.5]{hen80}. 
	
	We recall the definition of the Herbrand function.

	
	\begin{definition}\rm{\cite[\S IV.3]{ser79}}\label{Herbrand}
		Let $E$ be a finite extension of $F$ with ramification index $e(E/F)$. The Herbrand function $\phi_{E/F}:\mathbb{R}_{\geq 0}\rightarrow \mathbb{R}_{\geq 0}$ associated to $E/F$ is given by
		\[\phi_{E/F}(t)=1/{e(E/F)}\,\int_{0}^{t}\mid G_x\mid dx.\]
	\end{definition}

    Now we define conjugate representation. Let $E/F$ be a finite Galois extension and $\tau$ be a representation of $\mc W_E$. Then $W_E$ is a normal subgroup of $\mc W_F$ and $\mc W_F/\mc W_E\simeq {\rm Gal}(E/F)$. Thus every $\lambda\in {\rm Gal}(E/F)$ has a lift in $\mc W_F$, which we again denote by $\lambda$. Fixing such a $\lambda$, we define a conjugate representation $\tau^{\lambda}$ of $\mc W_E$ by \[\tau^{\lambda}(x)=\tau(\lambda x \lambda^{-1})\] for $x\in \mc W_E$. One can check that the slope of a Galois representation is invariant under Galois conjugate and its dual. Thus for a finite Galois extension $E/F$ and  an irreducible representation $\tau$ of $\mathcal{W}_E$, we have
       \[{\rm sl}(\tau)={\rm sl}(\tau^{\lambda})={\rm sl}((\tau^{\lambda})^{\vee})={\rm sl}(\tau^{\vee}) \]
    for $\lambda\in{\rm Gal}(E/F)$.

    The following proposition says that the slopes of the induced and inducing representations are related by the associated Herbrand function.
    \begin{proposition}{\rm \cite[Proposition 2]{mp19}}\label{slope-induction}
       Let $F\subset E \subset K$ be finite Galois extensions with $G(K/F):={\rm Gal}(K/F)$ and $G(K/E):={\rm Gal}(K/E)$. Let $\tau$ be a representation of $G(K/E)$. Then 
       \[{\rm sl}({\rm Ind}_{G(K/E)}^{G(K/F)}\tau)=\phi_{E/F}({\rm sl}(\tau)),\]
       where $\phi_{E/F}$ is the Herbrand function associated to the extension $E/F$.
    \end{proposition}

    The behavior of the slope of a representation under the tame induction and restriction is given in the next lemma.
    \begin{lemma}\label{slope-restriction-induction}
       Let $E/F$ be a finite totally tamely ramified Galois extension with ramification index $e$. Let $\sigma$ and $\tau$ be finite dimensional representations of $\mathcal{W}_F$ and $\mathcal{W}_E$ respectively, then
          \begin{enumerate}
	        \item[\rm(1)] ${\rm sl}(\sigma\mid_{{\mathcal{W}_E}})=e\cdot {\rm sl}(\sigma),$
	        \item[\rm(2)] ${\rm sl}({\rm Ind}_{E/F}\tau)=\dfrac{1}{e}\cdot{\rm sl}(\tau).$
          \end{enumerate} 
    \end{lemma}

    \begin{proof}
     We denote $G=\mathcal{W}_F$ and $H=\mathcal{W}_E$. Then  $G/H\simeq {\rm Gal}(E/F)$, the Galois group of $E/F$. Since $E/F$ is finite and tamely ramified, we have $G^u=H^{eu}$ for all $u\in \mathbb{R}_{>0}$ (cf. \cite[Lemma 1]{bus-hen17b}). 

     We first prove $(1)$. Let ${\rm sl}(\sigma)=u$. Then $\sigma\mid_{H^{eu+e\epsilon}}=\sigma\mid_{G^{u+\epsilon}}=1$ for all $\epsilon>0$. This shows that ${\rm sl}(\sigma\mid_H)\leq eu=e\cdot {\rm sl}(\sigma)$. Conversely, let ${\rm sl}(\sigma\mid_H)=u$. For any $\epsilon>0,$ we have
     $ \sigma\mid_{G^{u/e+\epsilon/e}} = \sigma\mid_{H^{u+\epsilon}} = 1$, which gives ${\rm sl}(\sigma)\leq u/e$, i.e., $e\cdot{\rm sl}(\sigma)\leq u = {\rm sl}(\sigma\mid_H)$.
     Thus $(1)$ follows.
  
     The proof of $(2)$ follows from Proposition \ref{slope-induction}, by observing that $\phi_{E/F}(x)=x/e$ for a finite tamely ramified extension $E/F$.
    \end{proof}


    The following lemma describes the behavior of the slope of the adjoint representation under the tame restriction. 
       \begin{lemma}\label{tame restriction formula}
        Let $E/F$ be a finite tamely ramified extension with the ramification index $e$. Let $\tau$ be a representation of $\mathcal{W}_F$. Then ${\rm sl}(\tau\otimes \tau^{\vee})=1/e\cdot {\rm sl}(\tau\mid_{\mathcal{W}_E}\otimes (\tau\mid_{\mathcal{W}_E})^{\vee})$.
       \end{lemma}
       \begin{proof}
          The proof follows from Lemma \ref{slope-restriction-induction}.
       \end{proof}

       \begin{lemma}\label{lem slope lower jump}
           Let $E/F$ be a finite Galois extension. Let $\rho={\rm Ind}_{E/F}\chi$ be an irreducible representation of $\mc W_F$ where $\chi$ be a character of $\mc W_E$. 
        Then the largest lower jump in ${\Gal}(E/F)$ is bounded above by ${\rm sl}(\chi)$.
        \end{lemma}
        \begin{proof}
           Since $\rho={\rm Ind}_{E/F}\chi$, there exists a finite Galois extension $K/F$ such that $\rho={\rm Ind}_H^G\chi$ where $G={\Gal}(K/F)$ and $H={\Gal}(K/E)$. Then we have ${\rm sl}(\rho)=\phi_{K/F}(m)$ and ${\rm sl}(\chi)=\phi_{K/E}(n)$ where $m$ and $n$ are the largest lower jumps in $G$ and $H$ respectively \cite[\S 2.1.2, p. 156]{hen-oi24}. Taking $\tau = \chi$ in Proposition \ref{slope-induction}, we have
           $$\phi_{K/F}(m)={\rm sl}(\rho)=\phi_{E/F}({\rm sl}(\chi))=\phi_{E/F}(\phi_{K/E}(n))=\phi_{K/F}(n).$$
           This gives $m=n$. By \cite[Lemma 5]{ser79} we have $${\Gal}(E/F)_{\phi_{K/E}(n)}=(G/H)_{\phi_{K/E}(n)}=G_nH/H.$$  
           For $\epsilon>0$ arbitrary, we have 
           $${\Gal}(E/F)_{\phi_{K/E}(n)+\epsilon}={\Gal}(E/F)_{\phi_{K/E}(n+\delta)}=G_{n+\delta}H/H=G_{m+\delta}H/H=1.$$
           Thus the largest lower jump of $G(E/F)$ must be less than or equal to $\phi_{K/E}(n) = {\rm sl}(\chi)$.
        \end{proof}

\section{Herbrand functions}\label{sec herbrand}
Let $E/F$ be a finite separable extension. Suppose $E/F$ is Galois, then the function $\psi_{E/F}: [0, \infty)\rightarrow [0, \infty)$ is defined by (cf. \cite[Chapter IV, p. 74]{ser79}):
\[\psi_{E/F}(x):=\int_{0}^{x}[G^0: G^t]~dt.\]
If $E/F$ is not Galois, take a Galois extension $E_1/F$ with $E\subset E_1,$ define (cf. \cite[(1.1.1)]{bus-hen19})
\[\psi_{E/F}=\psi_{E_1/E}^{-1}\circ \psi_{E_1/F}.\]
Then $\psi_{E/F}$ is well-defined. Indeed, let $E_2/F$ be an another Galois extension with $E\subset E_2$. 
We denote $\psi_{E/F}'=\psi_{E_2/E}^{-1}\circ \psi_{E_2/F}$. Since the composition of 
Galois extensions is Galois and $E_1E_2\supset E$, we have
$\psi_{E/F}''=\psi_{E_1E_2/E}^{-1}\circ \psi_{E_1E_2/F}$, which by using the fact that $\psi_{E_1E_2/E}=\psi_{E_1E_2/E_1}\circ\psi_{E_1/E}$ and $\psi_{E_1E_2/F}=\psi_{E_1E_2/E_1}\circ \psi_{E_1/F}$, equals
$\psi_{E/F}$. Again, using $\psi_{E_1E_2/E}=\psi_{E_1E_2/E_2}\circ\psi_{E_2/E}$ and $\psi_{E_1E_2/F}=\psi_{E_1E_2/E_2}\circ \psi_{E_2/F}$, we have $\psi_{E/F}''=\psi_{E/F}'$. Thus we have $\psi_{E/F}=\psi_{E/F}''=\psi_{E/F}'$.

\begin{definition}{\rm \label{psi jump}\cite[Definition, p. 1968]{bus-hen19}}
Let $E/F$ be a finite separable extension. We say that $a\in \mathbb{R}_{>0}$ is a jump of $\psi_{E/F}$, if the derivative $\psi_{E/F}'$ is not continuous at $a$.
\end{definition}

Following \cite[p. 1971]{bus-hen19}, for any $x>0$ we define
\begin{eqnarray}\label{psi-jump}
w_x(E/F):=\lim_{\epsilon\to 0}\dfrac{\psi_{E/F}'(x+\epsilon)}{\psi_{E/F}'(x-\epsilon)}.
\end{eqnarray}

\begin{remark}\label{psi-jump-condition}
Note that $x\in \mathbb{R}_{>0}$ is a jump of $\psi_{E/F}$ if and only if $w_x(E/F)> 1$.
\end{remark}

Although it should be well-understood, in the following lemma, we prove that 
jumps in a cyclic Galois group of order $p^r$ coincide with that of the associated Herbrand function.
\begin{lemma}\label{same set of jumps}
Let $E/F$ be a cyclic Galois extension of degree $p^r$ and let $G:={\rm Gal}(E/F)$.
The set of jumps of $\psi_{E/F}$ is same as the set of upper jumps in 
$G$ as defined in \S\ref{sec slope}. 
\end{lemma}

\begin{proof}
Since $G=\mathbb{Z}/{p^r\mathbb{Z}}$, by \cite[Chapter IV, p. 76]{ser79}, there exist positive integers $i_0,\dots, i_{r-1}$ such that \[J =\{j_k=i_0+\dots+ i_{k-1}\mid 1\leq k \leq r\}\] be the set of all upper jumps in $G$ with $|G^{j_k}|= p^{r-(k-1)}$ for $1\leq k\leq r$. 
For $j_k\leq x\leq j_{k+1}$, a short computation shows that
\[\psi_{E/F}(x)
= p^k x - \sum_{l=0}^{k-1}i_l(p^k - p^l).\]
Then by \eqref{psi jump}, we have
\[w_{j_k}(E/F)=\lim_{\epsilon\to 0}\dfrac{\psi_{E/F}'(j_k+\epsilon)}{\psi_{E/F}'(j_k-\epsilon)}=p^k/p^{k-1}=p>1,\] 
which by Remark \ref{psi-jump-condition} shows that
$\psi_{E/F}$ has a jump at $j_k$. On the other hand, for $j_k < x <j_{k+1}$, by \eqref{psi-jump} one sees that $w_x(E/F) = p^k/p^k =1$ and hence, by Remark \ref{psi-jump-condition} we have $\psi_{E/F}$ has no jump at $x$. 
Thus all jumps of $\psi_{E/F}$ occur at $j_k$ for $1\leq k \leq r$ which are exactly the jumps of $G$.
\end{proof}


The next proposition says that for large values of $x$, the Herbrand function $\psi_{E/F}(x)$ is easier to understand.
\begin{proposition}{\rm \cite[Proposition, p. 1970]{bus-hen19}}\label{psi-for-unique-jump}
Let $E/F$ be a separable and totally wildly ramified extension of degree $p^r$ with the largest upper jump $u$. Let $w_{E/F}$ be the wild exponent of $E/F$. Then, for $x\geq u$, we have
\[\psi_{E/F}(x)= p^rx-w_{E/F}.\]

\end{proposition}

\begin{proposition}\label{explict psi}
Let $E/F$ be a totally wildly ramified 
extension of degree $n$. Let $J=\{j_i \mid j_i < j_{i+1},  1 \leq i \leq m-1 \}$ be the set of all jumps of $\psi_{E/F}$. Then, there exists a unique sequence of fields $F=E_0\subset E_1 \subset \dots \subset E_m=E$ such that 
\[\psi_{E/F}(x)=\begin{cases}
	x & \text{if}~ x\leq j_1,\\
	\prod_{i=1}^k w_{j_i} x -w_{E_k/F} & \text{if}~j_k\leq x \leq j_{k+1},
\end{cases}\]
where $w_{j_i}$ is the degree of the extension $E_i/E_{i -1}$ and $w_{E_k/F}$ is the wild exponent of $E_k/F$.
\end{proposition}
\begin{proof}
This can be proved by \cite[Corollary 1]{bus-hen19} and repeated use of Proposition \ref{psi-for-unique-jump} together with \cite[Lemma (2), p. 1970]{bus-hen19}.
\end{proof}

\begin{remark}\label{explicit phi}
Let $E/F$ be a totally wildly ramified 
extension of degree $n$. Let $\phi_{E/F}$ denote the inverse of $\psi_{E/F}$ in Proposition \ref{explict psi}. Then we have 
\[\phi_{E/F}(x)=\begin{cases}
x & \text{if}~ x\leq j_1,\\
\dfrac{x}{\prod_{i=1}^k w_{j_i}} + \dfrac{w_{E_k/F}}{\prod_{i=1}^k w_{j_i}} & \text{if}~j_k\leq x \leq j_{k+1}.
\end{cases}\]
\end{remark}

\begin{corollary}\label{cyclic phi}
Let $E/F$ be a totally ramified cyclic 
extension of degree $p^r$. Then there exist positive integers $i_0, i_1,\dots, i_{r-1}$ with $j_k=i_0+i_1+\dots+i_{k-1}$ for $1\leq k\leq r$ such that
\[\phi_{E/F}(x)=\begin{cases}
x & \text{if}~ x\leq j_1,\\
\dfrac{x}{p^k} + \dfrac{w_{E_k/F}}{p^k} & \text{if}~j_k\leq x \leq j_{k+1}
\end{cases}.\]
\end{corollary}

\begin{proof}
Since $E/F$ is a totally wild cyclic Galois extension of of degree $p^r$, by \cite[Chapter IV, p. 76]{ser79} and Lemma \ref{same set of jumps}, jumps of $\psi_{E/F}$ occur at $r$ distinct positive integers, $j_1<j_2<\dots<j_r$ (say). By \cite[Corollary 1]{bus-hen19}, there exist unique sequence of fields $F=E_0\subset E_1 \subset \dots \subset E_r=E$ with $w_{j_i} = [E_i : E_{i-1}] = p$ for $1\leq i \leq r-1$ such that $\psi_{E_i/E_{i-1}}$ has a unique jump $\psi_{E_{i-1}/F}(j_i)$. Now taking $w_{j_i}=p$ for each $i$ in  Remark \ref{explicit phi} the result follows.
\end{proof}



\section{Proof of Theorems \ref{int thm slope adjoint wild} and \ref{int thm slope adjoint any wild}}\label{sec slope wild}

In this section, we compute the slope of the adjoint representation of a Carayol representation induced from a character and of dimension $p^r$.

\subsection{Some lemmas}\label{sec lemmas cyclic}

This subsection consists of a few lemmas which will be essential in proving the theorems.

\begin{lemma}\label{min jump}
Let $E/F$ be a finite Galois extension. We have \[\min_{\lambda\in{\Gal}(E/F)}\,v_E(\lambda)=i_0+1,\] where $i_0$ is the smallest lower jump of ${\rm Gal}(E/F)$. 
\end{lemma}
\begin{proof}
We write $G:={\rm Gal}(E/F)$. Since $i_0$ is the smallest lower jump of $G$, we have $v_E(\lambda) \geq i_0+1$ for all $\lambda\in G$. This shows that $ \min_{\lambda\in G}\,v_E(\lambda)\geq i_0+1$. This minimum is attained at $\lambda\in G_{i_0}\backslash G_{i_0+1}$ and hence, the lemma follows.
\end{proof}



The notations and the proof of the next lemma are inspired by \cite{kil19}.
\begin{lemma}\label{equality-check}
Let $E/F$ be a totally ramified cyclic Galois extension of degree $p^r$ with Galois group $G={\rm Gal}(E/F)$. Let $\rho = {\rm Ind}_{E/F}\chi$ be a Carayol representation of $\mc W_F$, where $\chi$ denotes a character of $W_E$. Let $\lambda \in G\setminus \{\id\}$. Let $\varsigma = {\rm Sw}(\chi)$ and $\delta = v_E(\lambda) - 1$ with $\varsigma > \delta$. Then, we have $\varsigma \not\equiv \delta \mod p$.
\end{lemma}
\begin{proof}
Let $\varsigma> \delta$. 
We show that $\varsigma\not\equiv \delta \mod p$.
Note that $E/F$ is a totally wild cyclic Galois extension of degree $p^r$ and, therefore by \cite[Example, p. 76]{ser79}, the lower jumps $l_1,l_2, \dots, l_r$ are all congruent to each other modulo $p$. So there exists an integer $n$ such that $l_j\equiv n \mod p$ for all $j$. 
By \cite[Lemma (2), p. 1970]{bus-hen19}, we have
\[\Sw({\Ind}_{E/F}\,\chi) = \Sw(\chi)+ w_{E/F},\] where $w_{E/F}$ denotes the wild exponent of $E/F$. This further using \cite[Lemma 1.2 (iii)]{kil19} gives 
\[\Sw({\Ind}_{E/F}\,\chi) = \Sw(\chi)+ \sum_{\lambda\in G\backslash \{\id\}}({v_E}(\lambda) -1).\] 
But each $v_E(\lambda) -1 = l_j$ for some $j$ and $l_j \equiv n \mod p$ for all $\lambda\in G\setminus \{\id\}$. Also, note that the cardinality of $G\setminus \{\id\}$ is $p^r-1$.
Thus modulo $p$, the above equation becomes
\[\Sw({\Ind}_{E/F}\,\chi) \equiv {\Sw}(\chi)-n.\] Since by our assumption $\Ind_{E/F}\chi$ is Carayol, we have $\Sw(\chi)\not\equiv n \mod p$. Again, for each $\lambda\in G\setminus\{\id\}$, the number $v_E(\lambda) -1$ being some $l_j$ for some $j$, is congruent to $n$ modulo $p$. Thus we conclude that 
$\Sw(\chi)\not\equiv v_E(\lambda) -1\mod p$ for all $\lambda\in G\setminus\{\id\}$.		
\end{proof}


Let $E/F$ be a cyclic Galois extension of degree $p^r$. We denote $G:={\rm Gal}(E/F)$. Recall that 
\[\{j_m\mid j_m=\sum_{k=1}^{m}i_{k-1}, ~1\leq m \leq r\}~~ \&~~ \{l_m\mid l_m=\sum_{k=1}^{m}i_{k-1}p^{k-1}, ~1\leq m \leq r\}\] denote the upper and lower jumps in $G$ 
respectively, where $i_m$ are positive integers (cf. \cite[Chapter IV, p. 76]{ser79}). 

With these notations, we have the following important lemma which says that the difference between the largest and smallest lower jumps is bigger than the largest upper jump.
\begin{lemma}\label{large slope}
We have  $l_r-l_1 \geq j_r$ for $r\geq 2$ and $p\neq 2$.
\end{lemma}
\begin{proof}[Proof of Lemma \ref{large slope}]
Let $e_F$ be the absolute ramification index of $E$, i.e., $v_F(p)=e_F$. 
We analyze $l_r-l_1 = l_r-i_0$ in two different cases.

\textbf{Case 1:} Suppose that $l_{r-1}\geq p^{r-2}e_F/(p-1)$. Then by \cite[Exercise 5.(f), p. 95]{fv02}, we have $l_r=l_{r-1}+p^{r-1}e_F$. Since $l_r=l_{r-1}+i_{r-1}p^{r-1}$, we have $i_{r-1}=e_F$. Also we note that $j_1=i_0\leq \dfrac{p e_F}{p-1}$ (cf. \cite[p. 140, l. -4]{wym69}). We have \[l_r-i_0=\sum_{k=2}^{r}i_{k-1}p^{k-1}.\] 
Hence,
\begin{eqnarray*}
(l_r-i_0)-j_r
&=&\sum_{k=2}^{r}i_{k-1}p^{k-1}- \sum_{k=1}^{r}i_{k-1}\\
&=&\sum_{k=2}^{r-1}i_{k-1}(p^{k-1}-1)+i_{r-1}(p^{r-1}-1)-i_0\\
&\geq& e_{F}(p^{r-1}-1)-\dfrac{pe_{F}}{p-1}\\
&=& e_{F}(p^{r-1}-1-\dfrac{p}{p-1})\\
&\geq& 0\\
\end{eqnarray*}
for $r\geq 2$ and $p\neq 2$. 
	\textbf{Case 2}: Suppose $l_{r-1}\leq p^{r-2}e_{F}/(p-1)$. Then by \cite[Exercise 5.(f), p. 95]{fv02}, we have 
	
	\begin{eqnarray}\label{case-2-inequality}
		(1+p(p-1))l_{r-1}\leq l_r \leq p^re_{F}/(p-1)-(p-1)l_{r-1}.
	\end{eqnarray}
	
	Note that, $l_r-i_0=\sum_{k=2}^{r}i_{k-1}p^{k-1}$ and $j_r=i_0+i_1+\dots+i_{r-1}$.
	Thus
	\begin{eqnarray*}
		l_r-i_0 & \geq & j_r\\
		\iff
		i_1p+i_2p^2+\dots+i_{r-1}p^{r-1} & \geq & i_0+i_1+\dots+i_{r-1} \\
		\iff (l_2-l_1)+(l_3-l_2)+\dots+(l_r-l_{r-1}) & \geq & l_1+(l_2-l_1)/p+
		\dots+(l_r-l_{r-1})/p^{r-1}\\
	\end{eqnarray*}
	which is true if and only if 
	\begin{eqnarray}\label{enough-to-show}
		(l_2-l_1)(1-1/p)+(l_3-l_2)(1-1/p^2)+\dots
		+(l_r-l_{r-1})(1-1/p^{r-1})&  \geq & l_1
	\end{eqnarray}
	
	
	We now analyze the left hand side of the above inequality. Note that
	by the left inequality in \eqref{case-2-inequality},
	we have $l_r-l_{r-1}\geq p(p-1)l_{r-1}$. Using this, we have
	\begin{eqnarray}\label{after-use-of-inequality}
		(l_2-l_1)(1-1/p)+
		\dots+(l_r-l_{r-1})(1-1/p^{r-1})\geq\nonumber\\
		(l_2-l_1)(1-1/p)+
		\dots+(l_{r-1}-l_{r-2})(1-1/p^{r-2})+p(p-1)l_{r-1}(1-1/p^{r-1}).\\ \nonumber
	\end{eqnarray}
	
	Observe that the coefficient of $l_{r-1}$ in \eqref{after-use-of-inequality} is
	\[1-1/p^{r-2}+p^2-p-1/p^{r-3}+1/p^{r-2}=1-1/p^{r-3}+p^2(1-1/p)=(1-1/p)(1+1/p+\dots+1/p^{r-4}+p^2).\] 
	Taking out a factor of $(1-1/p)$ 
	from 
	the right hand side of the inequality in
	\eqref{after-use-of-inequality}, we have
	\begin{align}\label{after-pull-out}
		&(l_2-l_1)+(l_3-l_2)(1+1/p)+(l_4-l_3)(1+1/p+1/p^2)+\dots \nonumber \\ \nonumber
		~~&+(l_{r-2}-l_{r-3})(1+1/p+\dots+1/p^{r-4})-l_{r-2}(1+1/p+\dots+1/p^{r-3})\\
		&~~~~+l_{r-1}(1+1/p+1/p^2+\dots+1/p^{r- 4}+p^2),\\\nonumber
	\end{align}
	which by writing $l_{r-1}=l_{r-2}+i_{r-2}p^{r-2}$, equals
	\begin{align*}
		&-l_1-l_2/p-\dots-l_{r-3}/p^{r-4}-l_{r-2}/p^{r-3}\\
		&~~~~~~+(l_{r-2}+i_{r-2}p^{r-2})(1+1/p+1/p^2+\dots+1/p^{r- 4}+p^2)\\
		&=-l_1-l_2/p-\dots -l_{r-4}/p^{r-5}-l_{r-3}/p^{r-4}\\
		&~~~~~~~+l_{r-2}(-1/p^{r-3}+1+1/p+1/p^2+\dots+1/p^{r- 4}+p^2)\\
		&~~~~~~~~~+i_{r-2}p^{r-2}(1+1/p+1/p^2+\dots+1/p^{r- 4}+p^2).\\
	\end{align*}
	
	Continuing this process, after a finite number of steps the last expression equals
	\begin{align}\label{reduced-4.3}
		&l_1(p^2-1/p^{r-3})+i_1p(-1/p^{r-3}+1+p^2)\nonumber\\\nonumber
		&~~~~~~~~~~~~~~~~~~~~~~~~~+i_2p^2(-1/p^{r-3}+1+1/p+p^2)\\\nonumber
		&~~~~~~~~~~~~~~~~~~~~~~~~~~+i_3p^3(-1/p^{r-3}+1+1/p+1/p^2+p^2)\\\nonumber
		&~~~~~~~~~~~~~~~~~~~~~~~~~~~+i_4p^4(-1/p^{r-3}+1+1/p+1/p^2+1/p^3+p^2)\\\nonumber
		&~~~~~~~~~~~~~~~~~~~~~~~~~~~~+\dots\\\nonumber
		&~~~~~~~~~~~~~~~~~~~~~~~~~~~~~+i_{r-3}p^{r-3}(-1/p^{r-3}+1+1/p+1/p^2+\dots+1/p^{r- 4}+p^2)\\\nonumber
		&~~~~~~~~~~~~~~~~~~~~~~~~~~~~~~+i_{r-2}p^{r-2}(1+1/p+1/p^2+\dots+1/p^{r- 4}+p^2)\\
		&=l_1(p^2-1/p^{r-3})+(*),\\\nonumber
	\end{align}
	where $(*)$ is a positive number.
	
	Now
	by \eqref{after-use-of-inequality},  \eqref{after-pull-out} and \eqref{reduced-4.3}, it follows that
	\[(l_2-l_1)(1-1/p)+
	+\dots+(l_r-l_{r-1})(1-1/p^{r-1})\geq l_1(1-1/p)(p^2-1/p^{r-3})+(*)\geq l_1,\] 
	which is \eqref{enough-to-show}. 
\end{proof}



\subsection{Proof of Theorem \ref{int thm slope adjoint wild}} \label{subsec slope adjoint wild}
Now we are ready to prove Theorem \ref{int thm slope adjoint wild}. For convenience we restate it below.

\begin{theorem}\label{thm slope adjoint wild} Let $\rho$ be a Carayol representation of $\mathcal{W}_F$ of dimension $p^r$ with $r>1$. Suppose
	$\rho=\Ind_{E/F}\,\chi$ for some totally ramified Galois extension $E/F$ of degree $p^r$ such that ${\rm Gal}(E/F)$ is cyclic and a character $\chi$ of $\mathcal{W}_E$. Let $i_0$
	denote the smallest lower jump of ${\Gal}(E/F)$. Then
	\[\slope(\rho\otimes\rho^\vee)=\slope(\rho)-\frac{i_0}{p^r}.\] 
\end{theorem} 
Before we proceed with the proof, here are some remarks.
\begin{remark}\label{positive slope}
	By the definition of Carayol representation, any one dimensional representation of $\mathcal{W}_F$ with positive Swan conductor 
	is Carayol. In Theorem \ref{thm slope adjoint wild}, the character $\chi$ is Carayol (cf. {\rm \cite[Proposition 7.2 (ii)]{kil19}}) and hence, ${\rm Sw}(\chi)>0$.
\end{remark}

\begin{remark}
	In Theorem \ref{thm slope adjoint wild}, the extension $E/F$ is totally and wildly ramified. So, by writing $G:={\rm Gal}(E/F)$, we have $G = G_0 = G_1$, and hence $i_0>0$. Thus $\slope(\rho\otimes\rho^\vee) < \slope(\rho)$ which is consistent with {\rm \cite[Proposition 3.4 (2)]{kil19}}.

\end{remark}




\begin{remark}\label{sanity check}
	We check that 
	\[\slope(\rho\otimes\rho^\vee)\geq {\rm Sw}(\rho\otimes\rho^\vee)/\, \dim (\rho\otimes\rho^\vee),\] 
	which shows that Theorem \ref{thm slope adjoint wild} is consistent with {\rm  \cite[Proposition 3.4 (iv)]{kil19}}. 
\end{remark}

\begin{proof}[Proof of Remark \ref{sanity check}]
	Since $\dim \,\rho = p^r$, it is enough to show that $p^{2r}\slope(\rho\otimes\rho^\vee) \geq {\rm Sw}(\rho\otimes\rho^\vee)$.
	By  Theorem \ref{thm slope adjoint wild} and {\rm \cite[Theorem 7.3]{kil19}}, we have
	\begin{eqnarray*}
		p^{2r}\slope(\rho\otimes\rho^\vee) - {\rm Sw}(\rho\otimes\rho^\vee)&=& p^{2r}\left(\slope(\rho)-\frac{i_0}{p^r}\right)-(p^r -1){\rm Sw}(\rho)\\
		&=& p^r\left({\rm sl}(\rho) -i_0\right)\\
		&=& p^r(\phi_{E/F}({\rm sl}(\chi)) - i_0)\\
		&\geq& p^r(\phi_{E/F}(l_r) - i_0)\\
		&=& p^r(j_r - i_0)\\
		& \geq & 0.\\
	\end{eqnarray*}
	The fourth inequality follows from Lemma  \ref{lem slope lower jump}. Finally, since $E/F$ is cyclic and Galois and $l_r$ is the largest lower jump of $G$, a small computation shows that $\phi_{E/F}(l_r)= j_r=i_0+i_1+\dots+i_{r-1}$, the largest upper jump of $G$ with $i_j$ are strictly positive integers {\rm \cite[Example, p. 76]{ser79}}. Thus the last equality follows. 
\end{proof}

\begin{proof}[Proof of Theorem \ref{thm slope adjoint wild}]
We denote $G:={\rm Gal}(E/F)$. By Mackey decomposition, we have
\[\rho\otimes\rho^\vee = \Ind_{E/F} \chi \otimes \Ind_{E/F}\chi^{-1} = \Ind_{E/F}\,\bigoplus_{\lambda\in G}\,\chi^{\lambda-1}=\bigoplus_{\lambda\in G}\, \Ind_{E/F}\,\chi^{\lambda-1}.\]
Here, by the local class field theory, we assume $\chi$ to be a character of $E^{\times}$ and $\chi^{\lambda -1}:E^{\times}\rightarrow \mathbb{C}^{\times}$ is defined by
$\chi^{\lambda -1}(x)=\chi(\dfrac{\lambda(x)}{x})$.
Hence $\chi^{\lambda- 1}=\chi^\lambda\chi^{-1}$. 

Now, by \cite[Proposition 3.4, p. 21]{kil19} we have
\begin{eqnarray}\label{sum-induction-slope}
	\slope(\rho\otimes\rho^\vee) 
	& = & {\max}_{\lambda\in G}\,\slope(\Ind_{E/F}\,\chi^{\lambda - 1})\nonumber\\ \nonumber
	& = & {\max}_{\lambda\in G}\,\phi_{E/F}\left(\slope(\chi^{\lambda - 1})\right)\\
	& = & \phi_{E/F}\left({\max}_{\lambda\in G}~\slope(\chi^{\lambda - 1})\right).\\\nonumber
\end{eqnarray}
The second equality above follows from \cite[Lemma 4.3]{aub-ply22}. Since $\phi_{E/F}$ is an increasing function, the third equality holds.

We now find \[{\max}_{\lambda\in G}~\slope(\chi^{\lambda - 1}).\]
For $\lambda \in G$, writing $\delta = v_E(\lambda) - 1$ and $\varsigma = {\rm sl}(\chi)$, by 
\cite[Lemma 4.5]{kil19} we have if $\varsigma \leq \delta$, then ${\rm sl}(\chi^{\lambda - 1}) = \Sw(\chi^{\lambda - 1}) = 0$. Since we are interested in the maximum of $\slope(\chi^{\lambda - 1})$, it is enough to concentrate on those $\lambda \in G$ for which $\varsigma > \delta$. Since $\rho = {\rm Ind}_{E/F}\chi$ is Carayol, for all $\lambda \in G$ with $\varsigma > \delta$, by Lemma \ref{equality-check} and 
\cite[Lemma 4.5]{kil19} we have 
\[\slope(\chi^{\lambda - 1}) = \varsigma - \delta = {\rm sl}(\chi) - (v_E(\lambda) - 1).\]
Thus by Lemma \ref{min jump}, we have 
\begin{eqnarray}\label{max-char-slope}
	{\max}_{\lambda\in G}~\slope(\chi^{\lambda - 1}) & = & {\rm sl}(\chi) - i_0,
\end{eqnarray}
where $i_0$ is the smallest lower jump in $G$.
Thus by \eqref{sum-induction-slope} and \ref{max-char-slope}, we have
\begin{eqnarray}\label{sl-is-phi}
	\slope(\rho\otimes\rho^\vee) = \phi_{E/F}({\rm sl}(\chi) - i_0).
\end{eqnarray}
Note that by Lemma \ref{large slope}, together with Lemma \ref{lem slope lower jump} we have ${\rm sl}(\chi)-i_0\geq j_r$ for $r\geq 2$. Therefore, by Corollary \ref{cyclic phi}, we have
\begin{eqnarray}\label{phi-at-large-value}
	\phi_{E/F}({\rm sl}(\chi)-i_0) & = &\dfrac{1}{p^r}\left({\rm sl}(\chi)-  i_0+w\right),\\\nonumber
\end{eqnarray}
where $w$ denotes the wild exponent of $E/F$.

Since $\rho = {\rm Ind}_{E/F}\chi$ is irreducible of dimension $p^r$, we have
\begin{eqnarray}\label{irr-swan-slope}
	{\rm sl}(\rho) & = & {\rm Sw}(\rho)/ \dim \rho \nonumber\\
	& = & 1/p^r(\Sw(\chi)+ w_{E/F}).
\end{eqnarray}
The first equality above follows from \cite[\S 2.1.2]{hen-oi24} and the second equality is obtained by \cite[Lemma (2), p. 1970]{bus-hen19}.

Finally, substituting the value of ${\rm sl}(\chi)$ obtained from \eqref{irr-swan-slope}, in \eqref{phi-at-large-value} and making use of \eqref{sl-is-phi}, we have
\[{\rm sl}(\rho\otimes\rho^\vee) = {\rm sl}(\rho) - i_0/p^r.\]
This completes the proof.
\end{proof}


\subsection{Proof of Theorem \ref{int thm slope adjoint any wild}}

Now we extend Theorem \ref{thm slope adjoint wild} for any totally wild extension.

\begin{theorem}\label{thm slope adjoint any wild}
Let $\rho$ be a Carayol representation of $\mathcal{W}_F$ of dimension $p^r$ with $r>1$. Suppose
$\rho=\Ind_{E/F}\,\chi$ for some totally ramified Galois extension $E/F$ of degree $p^r$ 
and a character $\chi$ of $\mathcal{W}_E$. Let $i_0$
denote the smallest lower jump of ${\Gal}(E/F)$. Then
\[\slope(\rho\otimes\rho^\vee)=\slope(\rho)-\frac{i_0}{p^r}.\] 
\end{theorem}

\begin{proof}[Proof of Theorem \ref{thm slope adjoint any wild}]
Note that ${\rm Gal}(E/F)$ is of order $p^r$, and hence it is supersolvable. Therefore there exists a sequence of Galois extensions such that $E=E_0\supseteq E_1\supseteq \cdots \supseteq E_{s-1} \supseteq F=E_s$ with ${\rm Gal}(E/F) \trianglelefteq {\rm Gal}(E/E_{s-1})\trianglelefteq 
\cdots \trianglelefteq {\rm Gal}(E/E_i)\trianglelefteq \cdots \trianglelefteq {\rm Gal}(E/E_1)\trianglelefteq \{e\}$ and $ {\rm Gal}(E/E_i)/ {\rm Gal}(E/E_{i-1}) \simeq {\rm Gal}(E_{i-1}/E_i)$ is cyclic of order $p^{t_i}$ for all $i = 1,2, \dots, s$. Let $\phi_i$ denote the Herbrand function for ${\rm Gal}(E_{i-1}/E_i)$, and $l_i$ and $l_i'$ denote the associated  largest and smallest lower jumps for $i=1,2,\dots,s$. Let $\phi_0$ denote the  Herbrand function of ${\rm Gal}(E/F)$ and  $l_0$ and $l_0'$ denote the associated largest and smallest lower jumps. We have 
\begin{eqnarray}\label{phi-decomposition}
	\phi_0 &=& \phi_s \circ\phi_{s-1}\cdots \phi_i\circ \phi_{i-1}\cdots \phi_2\circ \phi_1.
\end{eqnarray}
We prove that
\begin{eqnarray}\label{master-ineq}
	\phi_i'(l_0 - l_0')&\geq& l_i - l_i'
\end{eqnarray}
for $\phi_i'=\phi_{i-1}\circ \phi_{i-2}\circ\dots\circ \phi_2\circ\phi_1$ and $i=1,\dots,s$. 
For $i=1$, we use the convention that $\phi_1'$ denotes the identity map.

Note that $\phi_i'$ be the Herbrand function for ${\rm Gal}(E/E_{i-1})$. We denote its associated smallest and largest lower jumps by $l_{E/E_{i-1}}'$ and $l_{E/E_{i-1}}$ respectively. We also denote the smallest and largest lower jumps of ${\rm Gal}(E/E_i)$ by  $l_{E/E_i}'$ and $l_{E/E_i}$ respectively. 

Since ${\rm Gal}(E/E_{i-1})\subseteq {\rm Gal}(E/F)$, by Lemma \ref{sub-quo-lower-jump} we have 
\begin{eqnarray}\label{smallest-lower-jump-ineq}
	l_0' &\leq & l_{E/E_{i-1}}'.
\end{eqnarray}
Considering the Galois extensions $E\supseteq E_{i-1}\supseteq E_i$, we have 
\[\phi_{E/E_i} = \phi_{E_{i-1}/E_i}\circ \phi_{E/E_{i-1}}=\phi_i \circ \phi_i'.\]
By Lemma \ref{sub-quo-lower-jump} we have 
$\phi_i'(l_{E/E_{i-1}}') \leq l_i'$, which by \eqref{smallest-lower-jump-ineq} gives 
\begin{eqnarray}\label{left-part}
	\phi_i'(l_0') \leq \phi_i'(l_{E/E_{i-1}}') \leq l_i'.
\end{eqnarray}
We observe that 
\begin{eqnarray}\label{largest-lower-jump-ineq}
	l_i &\leq & \phi_i'(l_{E/E_i}).
\end{eqnarray}
Indeed, for any $\epsilon>0$, we have 
\begin{eqnarray*}
	{\rm Gal}(E_{i-1}/E_i)_{\phi_i'(l_{E/E_i})+\epsilon} 
	&=& \left({\rm Gal}(E/E_i)/{\rm Gal}(E/E_{i-1})\right)_{\phi_i'(l_{E/E_i})+\epsilon}\\
	&= &\left({\rm Gal}(E/E_i)/{\rm Gal}(E/E_{i-1})\right)_{\phi_i'(l_{E/E_i} + \delta)} \\
	&=& \left({\rm Gal}(E/E_i)_{l_{E/E_i} + \delta}{\rm Gal}(E/E_{i-1})\right)/{\rm Gal}(E/E_{i-1}) \\
	&= &1.
\end{eqnarray*}
The last but one equality follows from \cite[Lemma 5]{ser79}.
Since ${\rm Gal}(E/E_i) \subseteq {\rm Gal}(E/F)$, we have $l_{E/E_i}\leq l_0$. This by \eqref{largest-lower-jump-ineq} gives
\begin{eqnarray}\label{right-part}
	l_i\leq  \phi_i'(l_{E/E_i}) \leq  \phi_i'(l_0).
\end{eqnarray}
By \eqref{left-part} and \eqref{right-part}, we have
\[\phi_i'(l_0') \leq \phi_i'(l_{E/E_{i-1}}') \leq l_i' \leq  l_i\leq  \phi_i'(l_{E/E_i}) \leq  \phi_i'(l_0).\]
This gives
\[l_i -l_i'\leq\phi_i'(l_0) - \phi_i'(l_0') \leq \phi_i'(l_0) - l_0'\leq  l_0 - l_0'.\]
Thus we have
\[\phi_{E/E_i}(l_i -l_i')\leq\phi_{E/E_i}(l_0 - l_0').\]
Since $\phi_{E/E_i}=\phi_i\circ\phi_i'$ and $\phi_i\leq \phi_{E/E_i}\leq \phi_i'$, this implies that 
\[\implies \phi_i(l_i-l_i') \leq (\phi_i\circ \phi_i')( l_i -l_i')\leq (\phi_i\circ \phi_i')( l_0-l_0'),\]
which finally gives \eqref{master-ineq}
\[ l_i -l_i' \leq \phi_i'( l_0-l_0').\]
By Lemma \ref{lem slope lower jump} we have ${\rm sl}(\chi) \geq l_0$. So we have ${\rm sl}(\chi) - l_0' \geq l_0 -l_0'$ and hence, $\phi_i'({\rm sl}(\chi) - l_0') \geq \phi_i'(l_0 -l_0')$, which by the above inequality gives
\begin{eqnarray}\label{main-ineq}
	l_i -l_i'& \leq& \phi_i'({\rm sl}(\chi)-l_0'),
\end{eqnarray}
for $i=1,2,\dots,s$.

This further by making use of Lemma \ref{large slope} and Corollary \ref{cyclic phi} gives
\begin{eqnarray}\label{all-phi}
	\phi_i\left(\phi_i'({\rm sl}(\chi) - l_0')\right) &=& \dfrac{\phi_i'({\rm sl}(\chi) - l_0')}{p^{t_i}} + \dfrac{w_i}{p^{t_i}},
\end{eqnarray}
where $w_i$ denotes the wild exponent of $E_{i-1}/E_i$ for all $i = 1,\dots,s$.

Proceeding as in Theorem \ref{thm slope adjoint wild}, we have
\[\slope(\rho\otimes\rho^\vee) = \phi_{E/F}({\rm sl}(\chi) - l_0'),\]
where $l_0'$ is the smallest lower jump of ${\rm Gal}(E/F)$.
Recall that we denote $\phi_0=\phi_{E/F}$. So by \eqref{phi-decomposition} we have $\phi_0({\rm sl}(\chi) - l_0')$ equals
\begin{eqnarray*}\label{phi-computation}
	\phi_0({\rm sl}(\chi) - l_0') &=& \phi_s \circ\phi_{s-1}\cdots \phi_i\circ \phi_{i-1}\cdots \phi_2\circ \phi_1({\rm sl}(\chi) - l_0')\\
	&=& (\phi_s\circ\phi_s')({\rm sl}(\chi) - l_0'))\\
	&=& \phi_s\left(\phi_s'({\rm sl}(\chi) - l_0')\right)\\
	&=& \dfrac{\phi_s'({\rm sl}(\chi) - l_0')+w_s}{p^{t_s}}\\
	&=& \dfrac{w_s}{p^{t_s}} + \dfrac{1}{p^{t_s}}\phi_{s-1}\left(\phi_{s-1}'({\rm sl}(\chi) - l_0')\right)\\
	&=& \dfrac{w_s}{p^{t_s}} + \dfrac{1}{p^{t_s}}\left(\dfrac{\phi_{s-1}'({\rm sl}(\chi) - l_0')+w_{s-1}}{p^{t_{s-1}}}\right)\\
	&=&  \dfrac{w_s}{p^{t_s}} +  \dfrac{w_{s-1}}{p^{t_s+t_{s-1}}} + \dfrac{1}{p^{t_s+t_{s-1}}}\phi_{s-2}\left(\phi_{s-2}'({\rm sl}(\chi) - l_0')\right)\\
	&\vdots& \\
	&=& \dfrac{w_s}{p^{t_s}} +  \dfrac{w_{s-1}}{p^{t_s+t_{s-1}}} + \dots +\dfrac{w_{3}}{p^{t_s+t_{s-1}+\dots+t_3}} + \dfrac{1}{p^{t_s+t_{s-1}+\dots+t_3}}\phi_{2}\left(\phi_{2}'({\rm sl}(\chi) - l_0')\right)\\
	&=& \dfrac{w_s}{p^{t_s}} +  \dfrac{w_{s-1}}{p^{t_s+t_{s-1}}} + \dots +\dfrac{w_{2}}{p^{t_s+t_{s-1}+\dots+t_2}} + \dfrac{1}{p^{t_s+t_{s-1}+\dots+t_2}}\phi_{2}'({\rm sl}(\chi) - l_0')\\
	&=& \dfrac{w_s}{p^{t_s}} +  \dfrac{w_{s-1}}{p^{t_s+t_{s-1}}} + \dots +\dfrac{w_{2}}{p^{t_s+t_{s-1}+\dots+t_2}} + \dfrac{1}{p^{t_s+t_{s-1}+\dots+t_2}}\phi_{1}({\rm sl}(\chi) - l_0')\\
	&=& \dfrac{w_s}{p^{t_s}} +  \dfrac{w_{s-1}}{p^{t_s+t_{s-1}}} + \dots +\dfrac{w_{1}}{p^{t_s+t_{s-1}+\dots+t_1}} + \dfrac{{\rm sl}(\chi) - l_0'}{p^{t_s+t_{s-1}+\dots+t_1}}\\
	&=& \dfrac{1}{p^{\sum_{i=1}^{s}t_i}}\left(p^{\sum_{i=1}^{s-1}t_i}w_s + p^{\sum_{i=1}^{s-2}t_i}w_{s-1}+\dots+ p^{t_1}w_2 + w_1 + {\rm sl}(\chi) - l_0' \right)\\
	&=& \dfrac{w+{\rm sl}(\chi)-l_0'}{p^r}\\
	&=& \dfrac{p^r {\rm sl}(\rho) -l_0'}{p^r}\\
	&=& {\rm sl}(\rho) - \dfrac{l_0'}{p^r}.\\
\end{eqnarray*}
Note that $l_0'$ is denoted by $i_0$ in the statement of the theorem. The last but third equality follows from \cite[Lemma]{bus-hen19} (1) by denoting $w$ as the wild exponent of $E/F$. The last but second equality holds true by \cite[Lemma]{bus-hen19} (2).
\end{proof}

\section{Proof of Theorems \ref{int thm slope adjoint tame} and \ref{int thm slope carayol}}\label{sec slope gen}
In this section, we compute the slope of the adjoint representation of any Carayol representation (see Theorems \ref{int thm slope adjoint tame} and \ref{int thm slope carayol} in the introduction). Let us assume that $\rho$ is a Carayol representation of $\mc W_F$ of dimension $mp^r$ with $(m,p)=1$.


\subsection{Some lemmas}
This section consists of a series of lemmas required to prove Theorems \ref{int thm slope adjoint tame} and \ref{int thm slope carayol}. 

A representation of $\mathcal{W}_F$ is said to be \textbf{primitive} if it cannot be obtained by induction from a proper subgroup of $\mathcal{W}_F$ \cite[\S 3]{bus-hen14}. For example, every epipelagic representation of degree $p^r$ is primitive (cf. \cite[\S 3 Proposition]{bus-hen14}. 
We state the following important lemma on primitive representations. 
\begin{lemma}{\rm\cite[p. 37  l. -4, -2]{kil19}}\label{inducing character}
Let $\sigma$ be a Carayol primitive representation of $\mathcal{W}_F$ of degree $p^r$. Then  
\begin{enumerate}
	\item there exists a tamely ramified extension $T/F$ such that $\sigma\mid_{\mc W_T}$ is irreducible.
	\item there exists a totally wildly ramified finite Galois extension $K/T$ and a character $\chi$ of $\mc W_K$ such that $\sigma\mid_{\mc W_T}={\rm Ind}_{K/T}\chi$. 
\end{enumerate} 
\end{lemma}

The next lemma says that any irreducible representation $\rho$ of $\mc W_F$ of dimension $p^r$ is either primitive or induced from a primitive representation.
\begin{lemma}\label{p-power-reps} 
Let $\rho$ be an irreducible representation of $\mc W_F$ of dimension $p^r$. Then either $\rho$ is primitive or $\rho={\rm Ind}_{E/F}\tau$ where $\tau$ is a primitive representation of $\mc W_E$ and $E/F$ is a finite totally wild 
extension.
\end{lemma}

\begin{lemma}\label{same-tame-jump}
Let $E/T/F$ be Galois extension of fields with $E/T$ totally wild of degree $p^r$ and $T/F$ totally tame. Then, the smallest nonzero lower jumps in $G(E/T)$ and $G(E/F)$ are the same.
\end{lemma}

\begin{proof}
Let $G=G(E/F)$ and $H=G(E/T)$. Let $i_0$ be the 
smallest nonzero lower jump in $H$. Then $H_{i_0}\neq H_{i_0+1}$, i.e., $G_{i_0}\cap H\neq G_{i_0+1}\cap H$, which shows that $G_{i_0}\neq G_{i_0+1}$. So $i_0$ is a lower jump in $G$. It remains to show that $i_0$ is the smallest nonzero lower jump in $G$. Suppose not. There exists $0\neq j<i_0$ such that $G_j\neq G_{j+1}$. Since $T/F$ is tame, we have $(G/H)_1 = (G(T/F))_1 = 1$.
Since $E/T$ is totally wild, we have $i_0\geq 1$ and, hence $\phi_{E/T}(1)=1$.
So we have (cf. \cite[Lemma 5]{ser79})
\[1 = (G/H)_1 = (G/H)_{\phi_{E/T}(1)}=G_1H/H,\]
which gives $G_1\subset H$. Now $$H_k=H\cap G_k=H\cap (G_1\cap G_k)=(H\cap G_1)\cap G_k=G_k$$
for all $k\neq 0$. Thus $G_j\neq G_{j+1}$ implies $H_j\neq H_{j+1}$, which is a contradiction as $j<i_0$.
\end{proof}

\begin{lemma}\label{same-tame-slope}
Let $K/T/E/F$ be Galois extension of fields where $K/T$ and $E/F$ are 
totally wild but $T/E$ is a tame extension. Let $\chi$ be a character of $\mc W_K$. Then
\[\max_{\lambda\in G(E/F), \mu\in G(K/T)}{\rm sl}(\chi^{1-\lambda\mu})=\max_{\epsilon\in G(K/F)_1}{\rm sl}(\chi^{1-\epsilon}).\]
\end{lemma}

\begin{proof}
Noting that $G(K/F)/G(K/E)\simeq G(E/F)$, we have $G(K/F)=\sqcup_{i} \lambda_iG(K/E)$ where $\lambda_i\in G(K/F)$ is a lift of the corresponding element in $G(E/F)$. 
Similarly, denoting the lifts of the elements 
of $G(T/E)$ in $G(K/E)$ by $\varsigma_j$, we have $G(K/E)=\sqcup_{j}\varsigma_jG(K/T)$. Thus \[G(K/F)=\sqcup_{i, j} \lambda_i\varsigma_jG(K/T).\] 
Since $T/E$ is a tame extension, each of $\varsigma_j$ has valuation $1$ and hence, the same is true for every element of the form $\lambda_i\varsigma_j\mu$ where $\mu
\in G(K/T)$. Thus elements in $G(K/F)$ with valuation strictly bigger than one are obtained by taking $\varsigma_j=\id$ in the above decomposition. Therefore 
\[\{\lambda\mu\mid \lambda\in G(E/F),\mu\in G(K/T)\} = \{\epsilon\in G(
K/F)\mid v_K(\epsilon)\geq 2\}=G(K/F)_1.\]
Hence the lemma follows. 
\end{proof}

\begin{lemma}\label{sub-quo-lower-jump}
Let $K/E/F$ be finite Galois extension of fields. Let $H=G(K/E)$ and $G=G(K/F)$. Let $i,i'$ and $i''$ be the smallest nonzero lower jumps of $G, H$ and $G/H$ respectively. Then we have
\[i'\geq i ~\text{and}~ i''\geq \phi_{K/E}(i).\]
\end{lemma}

\begin{proof}
We prove that $i'\geq i$. We have $H_s=G_s\cap H$ for all $s>0$. Since $i$ is the smallest nonzero lower jump in $G$, we must have $G_s=G_t$ for $0<s, t\leq i$. So $H_s=H_t$ for $0<s, t\leq i$ and hence, the smallest nonzero lower jump in $H$ must be at least $i$.

Next we show that $i''\geq \phi_{K/E}(i)$. Recall that $(G/H)_{\phi_{K/E}(k)}=G_kH/H$ for any real number $k$ (cf. \cite[Lemma 5]{ser79}). 
Let $0<s, t\leq \phi_{K/E}(i)$. Writing $s=\phi_{K/E}(j), t=\phi_{K/E}(l)$ and using the fact that $\psi_{K/E}$ is increasing, we have $0<j, l\leq i$. We have
$$(G/H)_s=(G/H)_{\phi_{K/E}(j)}=G_jH/H=G_lH/H=(G/H)_{\phi_{K/E}(l)}=(G/H)_t.$$
Thus for $0<s, t\leq \phi_{K/E}(i)$, we have $(G/H)_s=(G/H)_t$, and hence 
$i''\geq \phi_{K/E}(i)$.
\end{proof}

\begin{lemma} {\rm (\cite{kil19})} \label{same-tame-slope-char}
Let $E/F$ be a totally ramified finite Galois extension. Let $\chi$ be a character of $\mc W_E$. Let $G=G(E/F)$ and $\lambda\in G_0\setminus G_1$. Then $${\rm sl}(\chi^{1-\lambda})={\rm sl}(\chi).$$
\end{lemma}
\begin{proof}
Let $E^{\lambda}$ denote the fixed field of the cyclic subgroup $\langle \lambda \rangle \subset G$. Then $E/E^{\lambda}$ is a cyclic tamely ramified extension. We note that $1 - \lambda$ induces a map $U_E^l \rightarrow U_E^l$ for all $l\geq 1$ (\cite[Lemma 4.2 (ii)]{kil19}). Here $U_E^l$ denotes the unit subgroup of $E^*$. By class field theory, thinking of $\chi$ as a character on $E^*$, the result follows.
\end{proof}
\begin{remark}\label{rmk non-galois}
Since the ramification subgroups are defined for any finite separable extension (not necessarily Galois), following \cite[\S 1.1, p. 10]{kil19}, observe that, Lemma \ref{same-tame-slope-char} holds for any totally ramified finite extension.
\end{remark}

As mentioned in the introduction, any Carayol representation of $\mc W_F$ can be constructed from a character $\chi$ of $\mc W_K$ for some extension $K/F$. Let us briefly recall the construction that will be helpful in proving Theorems \ref{int thm slope adjoint tame} and \ref{int thm slope carayol}. Let $\rho$ be a Carayol representation of $\mathcal{W}_F$ of dimension $mp^r$ with $(p, m) =1$. There exists a totally and tamely ramified extension $F'/F$ of degree $m$ and a Carayol representation $\rho'$ of $\mathcal{W}_{F'}$ of degree $p^r$ such that $\rho={\rm Ind}_{F'/F}\rho'$. The representation $\rho'$ of $\mc W_{F'}$ is either induced from a Carayol representation of a subgroup of $\mc W_{F'}$ or primitive (i.e., not induced from a proper subgroup). If the representation $\rho'$ is further induced, there exists a totally wild extension $E/F'$ and a Carayol representation $\tau$ of $\mc W_{E}$ such that $\rho'=\Ind_{E/F'}\,\tau$. In this case $\tau$ is primitive. If $\rho'$ is primitive, we take $E=F'$. The representation $\tau$ being primitive implies that, there exists a tamely ramified extension $T/E$ such that $\tau\mid_{\mc W_T}$ is irreducible. 
	Finally, there exists a totally wildly ramified finite Galois extension $K$ of $T$ and a character $\chi$ of $\mc W_K$ such that $\tau\mid_{\mathcal W_T}={\rm Ind}_{K/T}\chi$. The following diagram might be useful: 
	
	\[
	\begin{tikzcd}
		\mathcal W_F \arrow{r} \arrow[swap]{d}{{\rm totally~ tame}} & \rho: {\rm dim} \,\rho=mp^r \\
		\mathcal W_{F'}  \arrow{r} \arrow[swap]{d}{{\rm totally~wild}} & \rho': {\rm dim}\, \rho'=p^r  \arrow{u}{{\rm induction}} \\
		\mathcal W_{E}  \arrow{r} \arrow[swap]{d}{{\rm tame}} & \tau: {\rm dim}\, \tau=p^s  \arrow{u}{{\rm induction}} \arrow{d}{{\rm restriction}} \\
		\mathcal W_T  \arrow{r} \arrow[swap]{d}{{\rm totally~ wild}} & \tau\mid_{\mathcal W_T}: {\rm dim}\, \tau\mid_{\mathcal W_T}=p^s \\
		\mathcal W_K \arrow{r} & \chi: {\rm dim}\, \chi=1 \arrow{u}{{\rm induction}}
	\end{tikzcd}
      \]

\subsection{Proof of Theorem \ref{int thm slope adjoint tame}}
The following proposition lies in the heart of the proof of Theorem \ref{int thm slope adjoint tame}.

\begin{proposition}\label{tame-slope-invariant}
Let $E/F$ be a totally ramified finite Galois extension. Let $\sigma$ be a 
Carayol representation of $\mc W_E$. Let $G={\rm Gal}(E/F)$ and $\lambda\in G_0\setminus G_1$. Then 
$${\rm sl}(\sigma\otimes (\sigma^{\vee})^\lambda)={\rm sl}(\sigma).$$
	\end{proposition}
	
	\begin{proof}[Proof of Proposition \ref{tame-slope-invariant}]
Let $\dim \,\sigma=mp^r$ with $(p, m)=1$. By \cite[Lemma 8.4, p. 358]{bus-hen03}
and \cite[Lemma 7.5, p. 20]{kil19}, there exist a tame Galois extension $K/E$ of degree $m$ such that $\sigma={\rm Ind}_{K/E}\tau$ where $\tau$ is a wildly ramified representation of dimension $p^r$. By Mackey decomposition we have
\begin{eqnarray}\label{general-val-one}
	\sigma\otimes(\sigma^{\vee})^{\lambda} 
	& = & {\rm Ind}_{K/E}\tau\otimes{\rm Ind}_{K/E}(\tau^{\vee})^{\lambda}\nonumber\\
	& = &\oplus_{\mu\in G(K/E)}{\rm Ind}_{K/E}(\tau\otimes (\tau^{\vee})^{\lambda\mu}).\\\nonumber
\end{eqnarray}
Since $K/E$ is tame, by Lemma \ref{slope-restriction-induction} we have
\begin{eqnarray}\label{tame-formula}
	{\rm sl}({\rm Ind}_{K/E}(\tau\otimes (\tau^{\vee})^{\lambda\mu}))=1/m\cdot {\rm sl}(\tau\otimes (\tau^{\vee})^{\lambda\mu}).\\\nonumber
\end{eqnarray}
Note that $\tau$ is either primitive or $\tau={\rm Ind}_{L'/K}\tau'$ where $L'/K$ is a finite Galois extension and $\tau'$ is primitive (cf. Lemma \ref{p-power-reps} and Remark \ref{rmk non-galois}). 

Suppose $\tau$ is primitive. There exists a tame 
extension $T/K$ such that $\tau\mid_{\mc W_T}$ is irreducible and $\tau\mid_{\mc W_T}={\rm Ind}_{L/T}\chi$, where $L/T$ is totally wild Galois extension and $\chi$ is a character of $\mc W_L$ (cf. Lemma \ref{inducing character}). Thus
\begin{eqnarray}\label{reduce-to-slope-tau}
	{\rm sl}(\tau\otimes (\tau^{\vee})^{\lambda\mu})&=&1/e(T/K)\cdot{\rm sl}(\tau\mid_{\mc W_T}\otimes (\tau\mid_{\mc W_T}^{\vee})^{\lambda\mu}) \nonumber\\
	&=& 1/e(T/K) \cdot {\rm sl}({\rm Ind}_{L/T}\chi\otimes {\rm Ind}_{L/T}\chi^{-\lambda\mu}) \nonumber\\
	&=& 1/e(T/K) \cdot {\rm sl}(\oplus_{\epsilon\in G(L/T)}{\rm Ind}_{L/T}\chi^{1-\lambda\mu\epsilon}) \nonumber\\
	&=& 1/e(T/K) \cdot \max_{\epsilon\in G(L/T)}\phi_{L/T}({\rm sl}(\chi^{1-\lambda\mu\epsilon})) \nonumber\\
	&=& 1/e(T/K) \cdot \phi_{L/T}({\rm sl}(\chi)) \nonumber\\
	&=& 1/e(T/K) \cdot {\rm sl}({\rm Ind}_{L/T}\chi) \nonumber\\
	&=& 1/e(T/K) \cdot {\rm sl}(\tau\mid_{\mc W_T}) \nonumber\\
	&=& {\rm sl}(\tau).\\ \nonumber  
\end{eqnarray}
Since $\lambda\in G_0\setminus G_1$, its lift in $G(L/F)$ (along the line of Remark \ref{rmk non-galois}) has valuation $1$. 
Therefore the lift of $\lambda\mu\epsilon$ in $G(L/F)$ has valuation $1$, i.e., $\lambda\mu\epsilon\in G(L/F)_0\setminus G(L/F)_1$. So by Lemma \ref{same-tame-slope-char}, the fifth equality follows.

Thus by \eqref{tame-formula} and \eqref{reduce-to-slope-tau}, each of the direct summand in $\eqref{general-val-one}$

has slope $1/m\cdot{\rm sl}(\tau)$, and therefore $${\rm sl}(\sigma\otimes(\sigma^{\vee})^{\lambda})=1/m\cdot{\rm sl}(\tau)={\rm sl}({\rm Ind}_{K/E}\tau)={\rm sl}(\sigma).$$

Now suppose $\tau$ is not primitive and so, $\tau={\rm Ind}_{L'/K}\tau'$ for some $L'$. Then 
\begin{eqnarray*}
	{\rm sl}(\tau\otimes (\tau^{\vee})^{\lambda\mu})&=& {\rm sl}({\rm Ind}_{L'/K}\tau'\otimes {\rm Ind}_{L'/K}(\tau'^{\vee})^{\lambda\mu})\\
	&=& {\rm sl}(\oplus_{\epsilon\in G(L'/K)}{\rm Ind}_{L'/K}(\tau'\otimes (\tau'^{\vee})^{\lambda\mu\epsilon}))\\
	&=& \max_{\epsilon\in G(L'/K)}\phi_{L'/K}({\rm sl}(\tau'\otimes (\tau'^{\vee})^{\lambda\mu\epsilon}))\\
	&=& \phi_{L'/K}({\rm sl}(\tau'))\\
	&=& {\rm sl}({\rm Ind}_{L'/K}(\tau'))\\
	&=& {\rm sl}(\tau).\\
\end{eqnarray*}
The fourth equality follows from the previous case because $\tau'$ is primitive.
Proceeding as in the previous case we have ${\rm sl}(\sigma\otimes(\sigma^{\vee})^{\lambda})={\rm sl}(\sigma$).
\end{proof}

\begin{proof}[Proof of Theorem \ref{int thm slope adjoint tame}]
Suppose $m>1$. Then $\rho={\rm Ind}_{F'/F}\rho'$ for some totally tamely ramified extension $F'/F$ of degree $m$ and wildly ramified representation $\rho'$ of $\mc W_{F'}$ of dimension $p^r$ (see \cite[Lemma 8.4, p. 358]{bus-hen03}).
By the tame reduction as in \cite[Lemma 7.5]{kil19}, without loss of generality we can assume that $F'/F$ is Galois.

Now
\begin{eqnarray}\label{gen-tame-prim}
{\rm sl}(\rho\otimes \rho^{\vee})
&=& {\rm sl}({\rm Ind}_{F'/F}\rho'\otimes {\rm Ind}_{F'/F}\rho'^{\vee}) \nonumber\\
&=& {\rm sl}(\oplus_{\lambda\in G(F'/F)}{\rm Ind}_{F'/F}(\rho'\otimes                (\rho'^{\vee})^{\lambda})) \nonumber\\
&=& \max_{\lambda\in G(F'/F)}{\rm sl}({\rm Ind}_{F'/F}(\rho'\otimes (\rho'^{\vee})^{\lambda})) \nonumber\\
&=& \max_{\lambda\in G(F'/F)}
1/e(F'/F)\cdot {\rm sl}(\rho'\otimes (\rho'^{\vee})^{\lambda}).\\ \nonumber
\end{eqnarray}
Since $F'/F$ is tame, for ${\rm Id}\neq \lambda\in G(F'/F)$, by Proposition \ref{tame-slope-invariant} we have 
\begin{eqnarray}\label{gen-tame-non-prim-non-id}
{\rm sl}(\rho'\otimes (\rho'^{\vee})^{\lambda}) &=& {\rm sl}(\rho').
\end{eqnarray}
On the other hand, for $\lambda={\rm Id}\in G(F'/F)$, we have (cf. \cite[Proposition 3.4]{kil19})
\begin{eqnarray}\label{dual-tensor-slope}
{\rm sl}(\rho'\otimes \rho'^{\vee}) \leq {\rm sl}(\rho').\\\nonumber
\end{eqnarray}
Finally, by \eqref{gen-tame-prim}, \eqref{gen-tame-non-prim-non-id} and \eqref{dual-tensor-slope}, we have 
\[{\rm sl}(\rho\otimes \rho^{\vee})=1/e(F'/F)\cdot{\rm sl}(\rho')={\rm sl}(\rho).\] This completes the proof of the theorem.
\end{proof}

\subsection{Proof of Theorem \ref{int thm slope carayol}}

Let us restate our second main theorem for convenience.
\begin{theorem}\label{thm slope carayol}
Let $\rho$ be a Carayol representation of $\mc W_F$ of dimension $p^r$. Then $${\rm sl}(\rho\otimes \rho^{\vee})={\rm sl}(\rho)-\dfrac{\phi_{K/F}(i_0)}{\dim \, \rho}$$ for the finite Galois extension $K/F$, where $i_0$ is the smallest nonzero lower jump in $G(K/F)$. 
\end{theorem}

\begin{proof}[Proof of Theorem \ref{thm slope carayol}]
We prove the theorem by analyzing the following two cases.

\textbf{Case 1:} Suppose $\rho$ is primitive. By Lemma \ref{inducing character}, there exists a tame extension $T/F$ such that $\rho\mid_{\mc W_T}$ is irreducible and $\rho\mid_{\mc W_T}={\rm Ind}_{K/T}\chi$ where $K/T$ is a totally wild Galois extension and $\chi$ is a character of $\mc W_K$. By Lemma \ref{tame restriction formula} we have 
$${\rm sl}(\rho \otimes \rho^{\vee})=
1/e(T/F) \cdot {\rm sl}(\rho\mid_{\mc W_T} \otimes \rho\mid_{\mc W_T}^{\vee}),$$
which by Theorem \ref{thm slope adjoint any wild} equals 
\[1/e(T/F) \cdot ({\rm sl}(\rho\mid_{\mc W_T}) -\dfrac{i_0(K/T)}{\dim \, \rho\mid_{\mc W_T}}),\] where $i_0(K/T)$ is the smallest nonzero lower jump in $G(K/T)$. This further gives
$$1/e(T/F) \cdot {\rm sl}(\rho\mid_{\mc W_T}) -\dfrac{i_0(K/F)}{e(T/F) \cdot p^r}={\rm sl}(\rho)-\dfrac{\phi_{K/F}(i_0(K/F))}{\dim \, \rho}.$$
The last equality holds because by Lemma \ref{same-tame-jump} we have
\[\phi_{K/F}(i_0(K/F))=
\phi_{T/F}(\phi_{K/T}(i_0(K/T)))=\phi_{T/F}(i_0(K/T))=\dfrac{i_0(K/T)}{e(T/F)}.\]

Now we analyze the case when $\rho$ is not primitive.

\textbf{Case 2:} By Lemma \ref{p-power-reps}, we have $\rho={\rm Ind}_{E/F}\tau$ where $\tau$ is a primitive representation of $\mc W_E$. 

If $\dim \, \tau=1$, then by Theorem \ref{thm slope adjoint any wild} the proposition follows. 

Now suppose $\dim\tau =p^s > 1$. Since $\tau$ is primitive, there exists a tame extension $T/E$
such that $\tau\mid_{\mc W_T}$ is irreducible and $\tau\mid_{\mc W_T}={\rm Ind}_{K/T}\chi$ where $K/T$ is a totally wild Galois extension and $\chi$ is a character of $\mc W_K$ (cf. Lemma \ref{inducing character}). 

Note that by Mackey decomposition we have
\[\rho\otimes\rho^{\vee}={\rm Ind}_{E/F}\tau\otimes {\rm Ind}_{E/F}\tau^{\vee}=\oplus_{\lambda\in G(E/F)}{\rm Ind}_{E/F}(\tau\otimes(\tau^{\vee})^{\lambda}).\]
By Proposition \ref{slope-induction}, from the above decomposition we have
\begin{eqnarray}\label{wild-non-prim-mackey1}
	{\rm sl}(\rho\otimes \rho^{\vee}) 
	&=& \max_{\lambda\in G(E/F)}\phi_{E/F}({\rm sl}(\tau\otimes(\tau^{\vee})^{\lambda}))\nonumber\\
	&=&\max_{\lambda\in G(E/F)}\phi_{E/F}(1/e(T/E) \cdot {\rm sl}(\tau\mid_{\mc W_T}\otimes((\tau\mid_{\mc W_T})^{\vee})^{\lambda})).\\\nonumber
\end{eqnarray}
Since $\tau\mid_{\mc W_T}={\rm Ind}_{K/T}\chi$, we have
\begin{eqnarray}\label{wild-non-prim-mackey2}
	\tau\mid_{\mc W_T}\otimes((\tau\mid_{\mc W_T})^{\vee})^{\lambda}&=&{\rm Ind}_{K/T}\chi\otimes {\rm Ind}_{K/T}\chi^{-\lambda}\nonumber\\
	&=&\oplus_{\mu\in G(K/T)}{\rm Ind}_{K/T}\chi^{1-\lambda \mu}.\\\nonumber
\end{eqnarray}
By making use of \eqref{wild-non-prim-mackey2} together with Proposition \ref{slope-induction}, we have
\begin{eqnarray}\label{wild-non-prim-tame-slope}
	1/e(T/E) \cdot {\rm sl}(\tau\mid_{\mc W_T}\otimes((\tau\mid_{\mc W_T})^{\vee})^{\lambda})&=&1/e(T/E) \cdot \max_{\mu\in G(K/T)}\phi_{K/T}({\rm sl}(\chi^{1-\lambda \mu}))\nonumber\\
	&=&\max_{\mu\in G(K/T)}1/e(T/E) \cdot \phi_{K/T}({\rm sl}(\chi^{1-\lambda \mu}))\nonumber\\
	&=& \max_{\mu\in G(K/T)} \phi_{T/E}(\phi_{K/T}({\rm sl}(\chi^{1-\lambda \mu}))\nonumber\\
	&=& \max_{\mu\in G(K/T)}(\phi_{K/E}({\rm sl}(\chi^{1-\lambda \mu})).\\\nonumber
\end{eqnarray}
Since $T/E$ is a tame extension, the third equality above follows by Lemma \ref{slope-restriction-induction}.

Now, by \eqref{wild-non-prim-mackey1} and \eqref{wild-non-prim-tame-slope} we have
\begin{eqnarray}\label{wild-non-prim-final-phi}
	{\rm sl}(\rho\otimes \rho^{\vee})
	&=&\max_{\lambda\in G(E/F),~\mu\in G(K/T)}\phi_{E/F}(\phi_{K/E}({\rm sl}(\chi^{1-\lambda \mu}))\nonumber\\
	&=&\max_{\lambda\in G(E/F),~ \mu\in G(K/T)}\phi_{K/F}({\rm sl}(\chi^{1-\lambda \mu}))\nonumber\\
	&=& \max_{\epsilon\in G(K/F)_1}\phi_{K/F}({\rm sl}(\chi^{1-\epsilon}))\nonumber\\
	&=&\phi_{K/F}({\rm sl}(\chi)-i_0)\\\nonumber
\end{eqnarray}
where $i_0$ is the smallest nonzero lower jump in $G(K/F)$ (cf. Lemma \ref{same-tame-slope}). 

To complete the proof, we need to compute $\phi_{K/F}({\rm sl}(\chi)-i_0)$.
Let $i_0'$ be the first nonzero lower jump of $G(K/T)$. Since $G(K/T) \subset G(K/F)$, we must have $i_0'\geq i_0$ (see Lemma \ref{sub-quo-lower-jump}(1)) and hence, ${\rm sl}(\chi)-i_0\geq{\rm sl}(\chi)-i_0'$. Since $K/T$ is a totally wild Galois extension of degree $p^s$,
following the proof of Theorem \ref{thm slope adjoint any wild} (computation of $\phi_0({\rm sl}(\chi) - l_0')$), we have
\[\phi_{K/T}({\rm sl}(\chi)-i_0)={\rm sl}({\rm Ind}_{K/T}(\chi))-i_0/p^s ={\rm sl}(\tau\mid_{\mc W_T}) -i_0/p^s.
\]
The last equality above follows by substituting $\tau\mid_{\mc W_T}={\rm Ind}_{K/T}\chi$.

Thus we obtain
\begin{eqnarray}\label{phi-kf-to-ef}
	\phi_{K/F}({\rm sl}(\chi)-i_0)&=& \phi_{T/F}(\phi_{K/T}({\rm sl}(\chi)-i_0))\nonumber\\
	&=& \phi_{T/F}({\rm sl}(\tau\mid_{\mc W_T}) -i_0/p^s)\nonumber\\
	&=& \phi_{E/F}(\phi_{T/E}({\rm sl}(\tau\mid_{\mc W_T}) -i_0/p^s))\nonumber\\
	&=& \phi_{E/F}(1/e(T/E) \cdot ({\rm sl}(\tau\mid_{\mc W_T}) -i_0/p^s))\nonumber\\
	&=& \phi_{E/F}({\rm sl}(\tau) -i_0/e(T/E) \cdot p^s).\\\nonumber
\end{eqnarray}
Since $\rho={\rm Ind}_{E/F}(\tau)$, where $E/F$ is a totally wild Galois extension of degree $p^{r-s}$, following the proof of Theorem \ref{thm slope adjoint any wild},
we have 
\begin{eqnarray}\label{wild-non-prim-phi-ef}
	\phi_{E/F}({\rm sl}(\tau) -t)&=& {\rm sl}({\rm Ind}_{E/F}(\tau))-t/p^{r-s}\nonumber\\
	&=&{\rm sl}(\rho)-t/p^{r-s},\\\nonumber
\end{eqnarray}
where $t$ is less than or equal to the smallest nonzero lower jump of $G(E/F)(\simeq G(K/F)/G(K/E))$, which by Lemma \ref{sub-quo-lower-jump} must be at least $\phi_{K/E}(i_0)$. 
So we have
\[\phi_{K/F}(i_0)=\phi_{E/F}(\phi_{K/E}(i_0))=\phi_{K/E}(i_0),\]
which further equals 
\[\phi_{T/E}(\phi_{K/T}(i_0))=\phi_{T/E}(i_0)=i_0/e(T/E).\]
The last equality follows because $T/E$ is a tame extension.

Thus by \eqref{phi-kf-to-ef} and \eqref{wild-non-prim-phi-ef}, we have
$$\phi_{K/F}({\rm sl}(\chi)-i_0)=\phi_{E/F}({\rm sl}(\tau) -i_0/e(T/E)\cdot p^s)={\rm sl}(\rho)-i_0/e(T/E)\cdot p^s \cdot p^{r-s}.$$
This shows
that
\begin{eqnarray}\label{phi-kf-final}
\phi_{K/F}({\rm sl}(\chi)-i_0)={\rm sl}(\rho)-\phi_{K/F}(i_0)/p^r.
\end{eqnarray}
Hence by \eqref{wild-non-prim-final-phi} and \eqref{phi-kf-final} we have
\[{\rm sl}(\rho\otimes \rho^{\vee})={\rm sl}(\rho)-\phi_{K/F}(i_0)/\dim \, \rho.\]
This completes the proof.
\end{proof}

\begin{remark}
Since one knows the explicit conductor formulas for the functorial lifts, e.g., conductor of pair, symmetric square lift, exterior square lift, Asai lift etc, we would like to get explicit formulas for the slope of these functorial lifts in future.
\end{remark}

\vspace{.03cm}
{\noindent \bf Acknowledgements:} The authors would like to thank Masao Oi for detailed comments on the initial draft of the article which has improved the article significantly.
They thank U. K. Anandavardhanan and Dipendra Prasad for helpful suggestions. The first author would like to thank National Board of Higher Mathematics (NBHM) for financial support and IISER Berhampur, India for providing the necessary facilities while carrying out the work. The second author would like to greatly appreciate the discussion with Gopal Prasad at Princeton University.

\bibliographystyle{amsalpha}
\bibliography{Ref-slope}
\end{document}